\documentclass[12pt]{amsart}

\usepackage{amssymb}
\usepackage{hyperref}
\usepackage{url}
\usepackage[all]{xy}
\usepackage{color}
\usepackage{graphicx}
\usepackage{float}

\newtheorem{thm}[equation]{Theorem}
\newtheorem{lem}[equation]{Lemma}
\newtheorem{cor}[equation]{Corollary}
\newtheorem{prop}[equation]{Proposition}
\newtheorem{conj}[equation]{Conjecture}

\newtheorem{heur}[equation]{Heuristic}

\newtheorem*{conj235}{Conjecture \ref{conj235}}
\newtheorem*{questionA}{Question A}
\newtheorem*{thmB}{Theorem B}

\theoremstyle{definition}
\newtheorem{rem}[equation]{Remark}

\newtheorem{defn}[equation]{Definition}
\newtheorem*{defn1}{Definition}

\newtheorem{exas}[equation]{Examples}

\numberwithin{equation}{section}

\def\Z{\mathbb{Z}}
\def\Q{\mathbb{Q}}

\def\pp{\mathbb{P}}
\def\R{\mathbb{R}}
\def\C{\mathbb{C}}

\def\bH{\mathbf{H}}

\def\P{\mathcal{P}}
\def\E{\mathcal{E}}

\def\cL{\mathcal{L}}

\def\X{\mathcal{C}}

\def\D{\mathcal{D}}

\def\cC{\mathcal{C}}

\def\Hom{\mathrm{Hom}}
\def\Gal{\mathrm{Gal}}
\def\rk{\mathrm{rank}}
\def\Prob{\mathrm{{\P}}}
\def\ord{\mathrm{ord}}

\def\End{\mathrm{End}}
\def\Sym{\mathrm{Sym}}

\def\cond{\mathrm{cond}}

\def\Var{\mathrm{Var}}

\def\FF{\tilde{F}}
\def\SSS{\tilde{S}}

\def\hookto{\hookrightarrow}
\def\onto{\twoheadrightarrow}

\def\ms#1{\{ #1 \}}

\def\bmu{\boldsymbol{\mu}}

\def\pile#1#2{\genfrac{}{}{0pt}{1}{#1}{#2}}

\title[Arithmetic conjectures inspired by modular symbols]{
Arithmetic conjectures suggested by the statistical behavior of modular symbols}

\author{Barry Mazur}
\address{Department of Mathematics, 
Harvard University,
Cambridge, MA 02138, 
USA}
\email{\href{mailto:mazur@math.harvard.edu}{mazur@math.harvard.edu}}
\author{Karl Rubin}
\address{Department of Mathematics, 
UC Irvine,
Irvine, CA 92697, 
USA}
\email{\href{mailto:krubin@uci.edu}{krubin@uci.edu}}


\thanks{This material is based upon work supported by the 
National Science Foundation under grants DMS-1302409 and DMS-1500316.}

\begin{document}

\begin{abstract}
Suppose $E$ is an elliptic curve over $\mathbb{Q}$ and $\chi$ is a Dirichlet character.
We use statistical properties of modular symbols to estimate heuristically the probability 
that $L(E,\chi,1) = 0$.  Via the Birch and Swinnerton-Dyer conjecture, this gives a 
heuristic estimate of the probability that the Mordell--Weil rank grows in abelian extensions  
of $\mathbb{Q}$. Using this heuristic we find a large class 
of infinite abelian extensions $F$ where we expect $E(F)$ to be finitely generated.

Our work was inspired by earlier conjectures (based on random matrix heuristics) 
due to David, Fearnley, and Kisilevsky. Where our predictions and theirs overlap, 
the predictions are consistent.
\end{abstract}

\maketitle

\tableofcontents

\part*{Introduction}
\label{int}

By a {\em number field} we will mean a field of finite degree over $\Q$, and any field 
denoted ``$K$" below will be assumed to be a number field.

\begin{defn1}  A variety $V$ over $K$ is said to be 
{\em diophantine stable for the extension $L/K$} if  $V(L)=V(K)$; i.e., 
if $V$ acquires no ``new'' rational points when extended from $K$ to $L$. Depending 
on emphasis intended, we will also sometimes say {\em $L/K$ is diophantine stable for $V$.} 
\end{defn1}

The ``minimalist philosophy'' leads us to the following question.

\begin{questionA}
\label{min} 
If $A$ is an abelian variety over $K$, $\ell$ is an odd prime number, and 
$m$ is a positive integer, is it the case that the cyclic Galois extensions 
of $K$ of degree $\ell^m$ that are diophantine stable for $A$ are of 
density $1$  (among all cyclic Galois extensions of $K$ of degree $\ell^m$, 
ordered by conductor)?
\end{questionA}  

The following result of \cite{MR} is an embarrassingly weak theorem in 
the direction of answering this question.

\begin{thmB} 
If $A/K$ is a simple abelian variety such that $\End_{\bar K}(A) = \End_{K}(A)$, 
then there is a set $S$ of prime numbers of positive density such that for all 
$\ell \in S$ and for all positive integers $m$ there are infinitely many cyclic 
Galois extensions of $K$ of degree $\ell^m$ that are diophantine stable for $A$.
\end{thmB}

Unfortunately we cannot replace the phrase ``infinitely many'' in the statement 
of Theorem B by ``a positive proportion'' (where ``proportion'' is 
defined by organizing these cyclic extensions by size of conductor). 
If we order the extensions that way and consider what we have proved for a 
given $\ell^m$, among the first $X$ of these we get at least $X/\log^{\alpha}X$ 
as $X\to \infty$ (for a small, but positive, $\alpha$).

Note that, for example, when $K/{\Q}$ is quadratic, the cyclic $\ell^m$ extensions 
of $K$ that are Galois dihedral extension of $\Q$  are potentially  a source of 
systematic diophantine instability but are of density $0$ in all cyclic 
$\ell^m$ extensions of $K$.

The aim of this paper is to 
study the distributions of values of modular symbols and of what we call 
{\em $\theta$-coefficients}, in order to develop a heuristic prediction of the 
probability of diophantine stability (or instability) in the specific case 
where $V=E$ is an elliptic curve, $K = \Q$, and $L/\Q$ is abelian.

For example, our heuristic leads to the following conjecture.

\begin{conj235}
Let $E$ be an elliptic curve over $\Q$ and $F/\Q$ any real abelian extension 
such that $F$ contains only finitely many subfields of degree $2, 3$ or $5$ over $\Q$. 
Then the group of $F$-rational points $E(F)$ is finitely generated.
\end{conj235}

For example, we can take the field $F$ in Conjecture \ref{conj235} to be 
the cyclotomic $\Z_\ell$-extension for any prime $\ell$, or the compositum of 
these $\Z_\ell$-extensions for all $\ell$, or the maximal abelian $\ell$-extension for any 
$\ell \ge 7$, or the compositum of all such extensions. 

Question A was inspired by conjectures of 
David, Fearnley and Kisi\-levsky \cite{DFK}, who deal specifically with the case of 
elliptic curves over ${\Q}$. They  conjecture (cf.\ Conjecture \ref{dfkconj} below) 
that for a fixed elliptic curve $E$ over $\Q$ and a fixed prime 
$\ell \ge 7$, there are only finitely many Dirichlet characters of order 
$\ell$ for which $L(E,\chi,1)=0$.    Consequently, (assuming the 
Birch--Swinnerton-Dyer Conjecture) only finitely many cyclic extensions of 
$\Q$ of order $\ell$ are diophantine unstable  for a fixed elliptic curve 
$E$ over ${\Q}$.    
Their conjectures are based on computations and random matrix heuristics.
 
In this paper we will consider a more naive heuristic,
that leads us to make conjectures first regarding the distributions 
of values of $\theta$-coefficients (see Sections \ref{theta} and \ref{dists}, 
and  Conjecture \ref{newspiky}).
These conjectures, supported by numerical evidence, lead us to Conjecture \ref{conj235}.

In the first part of the paper we introduce modular symbols, recall the 
properties we need, and define the $\theta$-coefficients 
(which are sums of modular symbols).
We use the known distribution properties of modular symbols, along with numerical calculations, 
to make conjectures about the (more mysterious) distributions of $\theta$-coefficients. 

We say ``more mysterious"  because (as suggested by random matrix heuristics)  
there may well be no useful limiting distribution, no matter how one normalizes 
the $\theta$-coefficients. Nevertheless, our heuristic needs only quite weak estimates, 
that we label ``upper bounds for heuristic likelihood."  These estimates seem to be in 
accord with the computations that we have made---and (it seems to us) by a comfortable margin.

In the second part of the paper we use these conjectures about $\theta$-coefficients 
to develop a heuristic for the probability of vanishing of twisted $L$-values 
$L(E,\chi,1)$ where $\chi$ is a Dirichlet character of order at least $3$.

\subsection*{Acknowledgements}
We would like to thank Jon Keating and Asbj{\o}rn Nordencroft for helpful 
conversations about the random matrix theory conjectures.

\part{Modular symbols}
   
\section{Basic properties of modular symbols}

Fix once and for all an elliptic curve $E$ defined over $\Q$.  
Since $E$ is fixed we will usually suppress it from the notation.  Let 
$N$ be the conductor of $E$, 
$$
\phi_E: \bar{\bH} \to  X_0(N)(\C) \to  E(\C)
$$ 
the (optimal) modular uniformization of $E$ by $\bar{\bH} := \bH \cup \pp^1(\Q)$, 
the completion of the upper half-plane $\bH := \{z \in \C \,|\, \mathrm{Im}(z) >0\}$,
$$
f_E(z) = \sum_n a_n e^{2\pi i n z}
$$ 
the newform associated to $E$, and
$$
\Omega^\pm_E = \Omega^\pm = \int_{\gamma^\pm}\omega_E
$$  
the real and imaginary periods of $E$, where  
$\omega_E$ is the N{\'e}ron differential of (the  N{\'e}ron model of) $E$, 
$\gamma^+$ is the appropriately oriented connected component of the real locus 
of $E$, and  $\gamma^-$   is, similarly, the appropriately oriented cycle in 
$E(\C)$ stabilized, but not fixed, by complex conjugation.

Note that we have two maps (one with ``$+$'', one with ``$-$'')
that we can evaluate on closed curves $\gamma$ in $E(\C)$, namely 
$$
\gamma \mapsto \frac{1}{\Omega^\pm}\biggl(\int_\gamma\omega_E 
   \pm \int_{-\gamma}\omega_E\biggr) \quad \in \Z
$$  
These induce linear maps
$$
H_1(E({\bf C}); \Z) \to \Z.
$$

The relationship of $\omega_E$ to $f_E$ is given  (cf.\ \cite{A-R-S}) by 
$$
\phi_E^*\omega_E = c_E\cdot  2\pi if_E(z)dz,
$$ 
where $c_E$ is a nonzero integer (``Manin's Constant").

For every $r \in \pp^1(\Q)$ the image of the ``vertical line" in 
the upper half-plane 
$$
\{z= r+iy \;|\; 0\le y \le \infty\}  \subset \bar {\bH}
$$  
in $E$ is an oriented compact curve that ``begins'' at $i\infty$ and terminates 
at some cusp in $E$ (and any cusp is of finite order by the Manin--Drinfeld Theorem).

\begin{defn}
For every $r \in \Q$ define the (raw) modular symbols 
$$
\ms{r} := 2 \pi i \int_{i \infty}^r f_E(z) dz \in \C
$$
and the plus/minus normalized modular symbols
$$
[r]^\pm := \frac{\ms{r} \pm \ms{-r}}{2\Omega^\pm}.
$$
We will also write simply $\Omega$ and $[r]$ for $\Omega^+$ and $[r]^+$, respectively.
\end{defn}

The modular symbols have the following well-known properties.

\begin{lem}
\label{manybullets}  
For every $r \in \Q$ we have: 
\begin{enumerate}
\item 
$[r]^\pm \in \delta_E^{-1}{\Z}$ for some positive integer $\delta_E$ independent of $r$
\item
$[r+1]^\pm = [r]^\pm =\pm [-r]^\pm$
\item  
If $A \in \Gamma_0(N) \subset {\rm SL}_2(\Z)$ then, viewing $A$ as a linear fractional 
transformation we have
$
[r]^\pm = [A(r)]^\pm - [A(\infty)]^\pm.
$
If further $A$ has a  complex (quadratic) fixed point, then 
$[A(\infty)]^\pm = 0$, and therefore
$$ 
[A(r)]^\pm = [r]^\pm
$$  
for all $r \in \Q \cup \{\infty\}$.

\item
{\bf Atkin--Lehner relation:}
Let $m \ge 1$ and write $N = ef$ where $f := \gcd(m,N)$. 
Assume that $e$ and $f$ are relatively prime, and let $W_e$ be the 
Atkin--Lehner Hecke operator \cite{A-L}.
Denote by   $w_e$  the eigenvalue of   $W_e$ on $f_E$.
If $a, d \in \Z$ and $ade \equiv -1 \pmod {m}$, then
$$
[d/m]^\pm = -w_e\cdot[a/m]^\pm.
$$
\item
{\bf Hecke relations: }
Suppose $\ell$ is a prime, and $a_\ell$ is the $\ell$-th Fourier coefficient of $f_E$. 
Then
\begin{enumerate}
\setcounter{enumi}{4}
\item 
If $\ell \nmid N$, then 
$
a_\ell\cdot  [r]^\pm = [\ell r]^\pm + \sum_{i=0}^{\ell-1} [(r+i)/\ell]^\pm.
$
\item
If $\ell \mid N$, then 
$
a_\ell\cdot  [r]^\pm = \sum_{i=0}^{\ell-1} [(r+i)/\ell]^\pm.
$
\end{enumerate}
\end{enumerate}
\end{lem}

\begin{proof}
For (i), we can take $\delta_E$ to be (any positive multiple of) $c_E\cdot \lambda_E$ 
where $\lambda_E$ is the l.c.m. of the orders of the image of the cusps of $X_0(N)$ in  $E$.
Assertion (ii) follows directly from the definition.   
 
The first part of assertion (iii) is evident. 
For the second part, let $z = A(z)$ be the fixed point  of $A$,
and $ \gamma$ a geodesic  from $\infty$ to $z$. By invariance under $A$ we get  
$$
\int_{\infty}^z f_E(z) dz = \int_{A(\infty)}^{A(z)} f_E(z) dz 
   = \int_{A(\infty)}^{z} f_E(z) dz. 
$$ 
Thus $\int_\infty^{A(\infty)} f_E(z) dz = 0$, and it follows that 
$[A(\infty)]^\pm=[\infty]^\pm = 0.$
   
For (iv), here is a construction of the Atkin--Lehner operator $W_e$.  
Let $f := \gcd(m,N)$ and $e := N/f.$ The $W_e$ operator is given by (any) 
matrix of the following form:
\begin{equation}
\label{we}
W_e :=   \left(\begin{array}{cc}
       -ae & b\\
       cN & de\end{array}\right),
\end{equation}
with $a,b,c,d \in {\bf Z}$  and $\det(W_e) = e.$

Let $c := -m/f$ and $b := (ade+1)/m \in \Z$.  
With these choices the matrix $W_e$ of \eqref{we} has determinant $e$, 
$$
W_e(\infty) = -\frac{ae}{cN} = -\frac{a}{cf} = \frac{a}{m},
$$
and (computing) 
$$
W_e(d/m)= \infty.
$$
Thus $W_e$ takes the path $\{\infty, d/m\}$ to the path $\{a/m, \infty\}$. It follows that 
$[d/m] = -w_e[a/m]$  where $w_e$ is the eigenvalue of $W_e$ acting on $f_E$.  
This is (iv), and the proof of (v) is straightforward.
\end{proof}

\section{Modular symbols and $L$-values}
\label{modsymb1}

\begin{defn}
Suppose $\chi$ is a primitive Dirichlet character of conductor $m$.  
Define the Gauss sum
$$
\tau(\chi) := \sum_{a = 1}^m \chi(a) e^{2\pi i a/m}
$$
and, if $L(E,s) = \sum a_n n^{-s}$, the twisted $L$-function
$$
L(E,\chi,s) := \sum_{n=1}^\infty \chi(n) a_n n^{-s}.
$$
\end{defn}

If $F/\Q$ is a finite abelian extension of conductor $m$, we will identify 
characters of $\Gal(F/\Q)$ with primitive Dirichlet characters of 
conductor dividing $m$ in the usual way.  

\begin{prop}
\label{bsdcor}
If $F/\Q$ is a finite abelian extension, and 
the Birch and Swinnerton-Dyer conjecture holds for $E_{/\Q}$ and $E_{/F}$, then
$$
\rk(E(F)) = \rk(E(\Q)) + \sum_{\pile{\chi : \Gal(F/\Q) \to \C^\times}{\chi\ne 1}}\ord_{s=1}L(E,\chi,s).
$$
\end{prop}

\begin{proof}
This follows from the identity
$$
L(E_{/F},s) = \prod_{\chi : \Gal(F/\Q) \to \C^\times} L(E,\chi,s).
$$
\end{proof}

\begin{thm}[Birch--Stevens]
\label{BS}
If $\chi$ is a primitive Dirichlet character of conductor $m$, then
$$
\sum_{a=1}^m \chi(a)[a/m]^\epsilon = \frac{\tau(\chi)L(E,\bar\chi,1)}{\Omega^\epsilon}.
$$
where the sign $\epsilon := \chi(-1)$ is the sign of the character $\chi$.
\end{thm}

\section{$\theta$-elements  and  $\theta$-coefficients}
\label{theta}
\begin{defn}
Suppose $m \ge 1$, and let $\Gamma_m = \Gal(\Q(\bmu_m)/\Q)$.  Identify $\Gamma_m$ with $(\Z/m\Z)^\times$ 
in the usual way, and let $\sigma_{a,m} \in \Gamma_m$ be the Galois automorphism corresponding to 
$a \in (\Z/m\Z)^\times$ (i.e., $\sigma_{a,m}(\zeta)=\zeta^a$ for $\zeta\in\bmu_m$).  
Define
$$
\theta_m^\pm:= \delta_E\sum_{a \in (\Z/m\Z)^\times} [a/m]^\pm\;\sigma_{a,m} \in \Z[\Gamma_m]
$$  
where $\delta_E$ is as in Lemma \ref{manybullets}(i).
  
If $F/\Q$ is a finite abelian extension of conductor $m$, so $F \subset \Q(\bmu_m)$, 
define the {\em $\theta$-element} (over $F$, associated to $E$) to be:
$$
\theta_F^\pm := \theta_m^\pm |_F\  \in\ \Z[\Gal(F/\Q)]
$$  
where  $\theta_m^\pm |_F$ is the image of $\theta_m^\pm$ under the natural 
restriction homomorphism 
$$
\Z[\Gal(\Q(\bmu_m)/\Q)] \to \Z[\Gal(F/\Q)].
$$
\end{defn}

By Lemma \ref{manybullets}(i) we have
\begin{equation}
\label{intval}
\theta_F^\pm = \sum_{\gamma \in \Gal(F/\Q)} c_{F,\gamma}^\pm\cdot  \gamma
  \in \Z[\Gal(F/\Q)]
\end{equation} 
where
$$
c_{F,\gamma}^\pm := \delta_E\sum_{\pile{a \in (\Z/m\Z)^\times}{\sigma_{a,m}|_F = \gamma}} [a/m]^\pm.
$$

We will refer to the $c_{F,\gamma}^\pm\in  \Z$ as {\em $\theta$-coefficients}.
Since we will most often be dealing with the ``plus''-$\theta$-elements, we will 
simplify notation by letting $\theta_F := \theta_F^+$,  $c_{F,\gamma} := c_{F,\gamma}^+$, 
and $\Omega := \Omega^+$.
If $F$ is a real field, then $\sigma_{-1,m}|_F = 1$, and $[a/m]=[-a/m]$, so
\begin{equation}
\label{tc}
c_{F,\gamma} = 2\delta_E\cdot \sum_{\pile{a \in (\Z/m\Z)^\times/\{\pm 1\}}{\sigma_{a,m}|_F = \gamma}} [a/m].
\end{equation}

With this notation, Proposition \ref{BS} can be rephrased as follows:

\begin{cor}
\label{thetaL}
Suppose $F/\Q$ is a finite real cyclic extension of conductor $m$ and 
$\chi :(\Z/m\Z)^\times \onto  \Gal(F/\Q) \hookto \C^\times$ is a character that cuts out $F$.
Then  
$$
{\bar \chi}(\theta_F) = \delta_E \frac{\tau(\bar\chi)L(E,\chi,1)}{\Omega}.
$$ 
In particular ${\bar \chi}(\theta_F)$ vanishes if and only if $L(E,\chi,1)$ does.
\end{cor}

\section{Distribution of modular symbols}
From now on, we assume that our elliptic curve $E$ is semistable, so its 
conductor $N$ is squarefree.

The following fundamental result about the distribution of modular symbols 
was proved by Petridis and Risager   (cf. (8.6) of  \cite{pr2}).

\begin{defn}
Let $\cC_E := 6/\pi^2 \prod_{\ell \mid N} (1 + \ell^{ -1 })^{ -1} L(\Sym^2(E),1).$  
\end{defn}

\begin{thm}[Petridis \& Risager \cite{pr2}, see also \cite{pr1}]
\label{pr1}
As $X$ goes to infinity the values 
$$
\biggl\{\frac{[a/m]^+}{\sqrt{\log(m)}} : m \le X,a \in (\Z/m\Z)^\times\biggr\}
$$ 
approach a normal distribution with variance $\cC_E$.  
\end{thm}

Numerical experiments led to the following conjecture.  Denote by $\Var(m)$ 
the variance
$$
\Var(m) := \frac{1}{\varphi(m)}\sum_{a \in (\Z/m\Z)^\times}([a/m]^+)^2
$$

\begin{conj}
\label{slopeshift}
\begin{enumerate}
\item
As $m$ goes to infinity, the distribution of the sets
$$
\biggl\{\frac{[a/m]^+}{\sqrt{\log(m)}} : a \in (\Z/m\Z)^\times\biggr\}
$$
converges to a normal distribution with mean zero and variance $\cC_E$.
\item
For every divisor $\kappa$ of the conductor $N$, there is a constant 
$\D_{E,\kappa} \in \R$ such that 
$$
\lim_{\pile{m \to \infty}{(m,N) = \kappa}}(\Var(m) -\cC_E\log(m)) = \D_{E,\kappa}.
$$
\end{enumerate}
\end{conj}

\begin{rem}
We expect that a version of Conjecture \ref{slopeshift} holds also for non-semistable 
elliptic curves, with a slightly more complicated formula for $\cC_E$.
\end{rem}

Note that Theorem \ref{pr1} is an ``averaged'' version of Conjecture \ref{slopeshift}(i).
Inspired by Conjecture \ref{slopeshift}, Petridis and Risager \cite[Theorem 1.6]{pr2} obtained 
the following result, which identifies the constant $\D_{E,\kappa}$ and 
proves an averaged version of Conjecture \ref{slopeshift}(ii).

\begin{thm}[Petridis \& Risager \cite{pr2}]
\label{pr2}
For every divisor $\kappa$ of $N$, there is an explicit 
(see \cite[(8.12)]{pr2}) constant $\D_{E,\kappa} \in \R$ such that 
$$
\lim_{X \to \infty} \frac{1}{\text{\tiny$\displaystyle\sum_{\pile{m < X}{(m,N_E) 
   = \kappa}}$}\varphi(m)}\sum_{\pile{m < X}{(m,N) = \kappa}}\varphi(m)(\Var(m) -\cC_E\log(m)) = \D_{E,\kappa}.
$$
\end{thm}

\begin{rem}
 Petridis \& Risager compute $\D_{E, \kappa}$  in terms of the values $L(\Sym^2(E),1)$ and 
$L'(\Sym^2(E),1)$. They deal with non-holomorphic Eisenstein series twisted by 
modular symbols. 
 
Lee and Sun \cite{LeeSun} more recently have  proven the same result 
(for arbitrary $N$, averaged over $m$, but without explicit determination of 
the constants $\cC_E$ and $\D_{E, \kappa}$) by considering 
dynamics of continued fractions.

A version of Conjecture \ref{slopeshift}(ii) with $m$ ranging over primes was proved by 
Blomer, Fouvry, Kowalski, Michel, Mili\'cevi\'c, and Sawin in 
\cite[Theorem 9.2]{Blomeretal}.

For related results regarding higher weight modular eigenforms see \cite{B-D}.
See also  \cite{constantinescu} and \cite{nordentoft} for other related results.
\end{rem}

\begin{rem}
The modular symbols are not completely ``random'' subject to Conjecture \ref{slopeshift}.   
Specifically, partial sums  $\sum_{a=\alpha}^{\beta}[a/m]$ behave in a somewhat orderly way.
Numerical experiments led the authors and William Stein to conjecture the following result, 
which was then proved by Diamantis, Hoffstein, Kiral, and Lee (for arbitrary $N$, with an 
explicit rate of convergence) \cite[Theorem 1.2]{DHKL}.

\begin{thm}[\cite{DHKL}]
If $0 < x < 1$ then
$$
\lim_{m \to \infty}\frac{1}{m}\sum_{a=1}^{mx}[a/m]
   = \sum_{n=1}^\infty\frac{a_n \sin(\pi nx)}{n^2 \Omega}
$$
where $\sum_n a_n q^n$ is the modular form $f_E$ corresponding to $E$.
\end{thm}

Another proof in the case of squarefree $N$ was given by Sun \cite[Theorem 1.1]{ks}.
\end{rem}

\section{The involution $\iota_F$}
\label{dt}
\begin{defn}
\label{sensitive}
Suppose $F$ is a finite real cyclic extension of $\Q$.  Let $m$ be the conductor of $F$,  
$f := \gcd(m,N)$, and $e := N/f$.  Since $N$ is assumed squarefree, $e$ is prime to $m$.  
Let $\gamma_F$ be the image of $e$ under the map 
$(\Z/m\Z)^\times \onto \Gal(F/\Q)$.  Define an involution $\iota_F$ of the set $\Gal(F/\Q)$ by
$$
\iota_F(\gamma) = \gamma^{-1} \gamma_F^{-1}
$$
Let $w_F = -w_e$ where $w_e$ is the 
eigenvalue of the Atkin--Lehner operator $W_{e}$ (see Lemma \ref{manybullets}(iv))
acting on $f_E$.
\end{defn}

Recall from \eqref{intval} 
that $\theta_F = \sum_{\gamma\in \Gal(F/\Q)} c_{F,\gamma} \gamma$.  

\begin{lem}
\label{alinv}
Suppose $F$ is a finite real cyclic extension of $\Q$.
\begin{enumerate}
\item
We have  
$c_{F,\gamma} = w_F c_{F,\gamma'}$ where $\gamma' = \iota_F(\gamma)$.
\item
The fixed points of $\iota_F$ are the square roots of $\gamma_F^{-1}$ in $\Gal(F/\Q)$, 
so the number of fixed points is:
\begin{itemize}
\item
one if $[F:\Q]$ is odd, 
\item
zero if $\gamma_F$ is not a square in $\Gal(F/\Q)$,
\item
two if $[F:\Q]$ is even and $\gamma_F$ is a square in $\Gal(F/\Q)$.
\end{itemize}
\item If $\gamma = \iota_F(\gamma)$ and $w_F=-1$, then $c_{F,\gamma} =0$.
\end{enumerate}
\end{lem}

\begin{proof}
Assertion (i) follows from the Atkin--Lehner relations satisfied by the modular symbols 
(Lemma \ref{manybullets}(iv)). Assertion (ii) is immediate from the definition, and (iii) 
follows directly from (i).
\end{proof}

\begin{defn}
\label{gendef}
If $F/\Q$ is a real cyclic extension,
we say that $\gamma \in \Gal(F/\Q)$ is {\em generic}, (resp., {\em special$^+$}, 
resp., {\em special$^-$}) if $\gamma \ne \iota_F(\gamma)$ 
(resp., $\gamma = \iota_F(\gamma)$ and $w_F=1$, resp., $\gamma = \iota_F(\gamma)$ and $w_F=-1$).
By Lemma \ref{alinv}(iii), if $\gamma$ is special$^-$ then $c_{F,\gamma} = 0$.
\end{defn}

\section{Distributions of $\theta$-coefficients}
\label{dists} 

Consider real cyclic extensions $F/\Q$ of fixed degree $d\ge 3$ and varying conductor $m$.
By \eqref{tc}, if $\gamma$ is generic (resp., special$^+$) then the $\theta$-coefficient $c_{F,\gamma}$ 
is $2 \delta_E$ times a sum of $\varphi(m)/(2d)$ modular symbols
(resp., $4 \delta_E$ times a sum of $\varphi(m)/(4d)$ modular symbols).
If these were {\em randomly chosen} modular symbols $[a/m]$, one would expect from 
Conjecture \ref{slopeshift}(i) that the collection of data
\begin{multline*}
\Sigma_d := 
\biggl\{\frac{c_{F,\gamma}\sqrt{d}}{\sqrt{\varphi(m)\log(m)}} : 
   \text{$F/\Q$ real, cyclic of degree $d$,} \\[-10pt]
    \text{$m = \cond(F)$, $\gamma \in \Gal(F/\Q)$ generic}\biggr\}
\end{multline*}
(ordered by $m$) would converge to  a normal distribution with variance $2 \delta_E^2 \cC_E$
as $m$ tends to $\infty$.
Similarly the data defined in the same way except with $\gamma$ special$^+$ 
instead of generic, would converge to a normal distribution with variance $4 \delta_E^2 \cC_E$.

Calculations do not support this expectation, at least not for small values of $d$.  See 
Examples \ref{data} below.  In addition, random matrix heuristics (see \cite{DFK,K-S}) 
predict that the data above for $d = 3$ will not converge to a non-zero distribution. 

However, calculations do support the following much weaker Conjecture \ref{newspiky} 
below, which is strong enough for our purposes.  

\begin{defn}
\label{6.1}
For every $F$ of degree $d$ and conductor $m$, and every $\gamma \in \Gal(F/\Q)$ 
define the {\em normalized $\theta$-coefficient}
$$
\tilde{c}_{F,\gamma} := \frac{c_{F,\gamma}\sqrt{d}}{\sqrt{\varphi(m)\log(m)}}.
$$

For $d \ge 3$, $\alpha, \beta \in \R$, and $X \in \R_{>0}$, 
let $\Sigma_{d,\alpha,\beta}(X)$ be the collection of data 
(counted with multiplicity)
\begin{multline*}
\Sigma_{d,\alpha,\beta}(X) := 
\{\tilde{c}_{F,\gamma}m^{\alpha} \log(m)^{\beta} : 
   \text{$F/\Q$ real, cyclic of degree $d$,} \\
    \text{$m = \cond(F) < X$, $\gamma \in \Gal(F/\Q)$ generic or special$^+$}\}.
\end{multline*}
\end{defn}

\begin{conj}
\label{newspiky}
There is a $B_E > 0$ and for every $d \ge 3$ there are $\alpha_d, \beta_d \in \R$ such that 
\begin{enumerate}
\item
for every real open interval $(a,b)$,
$$
\limsup_{X \to \infty} 
   \frac{\#\{\Sigma_{d,\alpha_d,\beta_d}(X) \cap (a,b)\}}{\#\Sigma_{d,\alpha_d,\beta_d}(X)} 
      < B_E(b-a),
$$
\item
$\{\alpha_d\varphi(d) : d \ge 3\}$ is bounded, and $\lim_{d\to\infty} \beta_d = 0$.
\end{enumerate}
\end{conj}

\begin{rem}
Random matrix theory heuristics (see \cite{DFK, DFK2, K-S}) suggest that for $d = 3$ 
and every real open interval $(a,b)$, 
\begin{equation}
\label{just}
\limsup_{X \to \infty} 
   \frac{\#\{\Sigma_{d,\alpha_d,\beta_d}(X) \cap (a,b)\}}{\#\Sigma_{d,\alpha_d,\beta_d}(X)} < B_d(b-a),
\end{equation}
with $\alpha_3 = 0$, $\beta_3 = 3/4$ and a sufficiently large $B_3$.
Empirical data (see Examples \ref{data}) suggest \eqref{just} holds for all $d$, 
with $\alpha_d = 0$, $\beta_d$ converging to zero for large $d$ and $B_d$ bounded for large $d$.  
Taking $B_E$ to be the maximum of the $B_d$ leads to the statement of Conjecture \ref{newspiky}.  
\end{rem}

\begin{rem}
\label{prob}
Here is a heuristic shorthand description of Conjecture \ref{newspiky}: 
There is a $B_{E} > 0$ such that for
\begin{itemize} 
\item
every cyclic extension $F/\Q$ of degree $d \ge 3$, 
\item  
every generic or special$^+$ $\gamma\in \Gal(F/\Q)$, and 
\item 
every real interval $(a,b)$,
\end{itemize}
the ``{\bf heuristic likelihood}" that $\tilde{c}_{F,\gamma}$ lies in $(a,b)$, which we will denote by 
$\Prob\bigl[\tilde{c}_{F,\gamma} \in (a,b)\bigr]$, is bounded above by
$$
B_{E} (b-a) m^{\alpha_d}\log(m)^{\beta_d}, 
$$ 
where  $m$ is the conductor of $F$ (and $\alpha_d, \beta_d$ are as in Conjecture \ref{newspiky}).

By {\em heuristic  likelihood ${\mathcal P}[X]$ of  an outcome $X$ occurring} 
we mean a real number that gives a (merely heuristically supported guess for an) 
upper bound for the number of times that the outcome $X$ occurs within  a specific 
range of events.  The main use we will make of such (nonproved, but only 
heuristically suggested) estimates  is that in various natural contexts  we produce 
computationally supported  upper bounds  of {\em heuristic likelihoods}  
${\mathcal P}[X_i]$ (for $i$ in some index set $I$) where the $\{X_i\}_{i\in I}$  
is a  collection of possible outcomes for some specific range of events. If  
$\sum_{i\in I}{\mathcal P}[X_i]$  is finite, we simply interpret our heuristic 
as suggesting that  the totality of outcomes  of the form $X_i$   (for $i \in I$)  
that actually occur---within the specific range of events---is finite.
\end{rem}
\begin{rem}\label{for.data}

The smaller those exponents $\alpha_d, \beta_d$ are---and the faster they 
approach $0$ as $d$ grows---the stronger is the claim in Conjecture \ref{newspiky}.

It may well be that $\alpha_d$ may be taken to be $0$ for all $d\ge 3$; the 
empirical data we have computed (see \ref{data} below) which  has led us to 
frame  Conjecture \ref{newspiky}, might almost suggest this. 
But we don't believe we have computed far enough yet to say anything 
that precise, and---more importantly---the eventual qualitative arithmetic 
conjecture that we make  (Conjecture \ref{conj235}) doesn't need such 
precision for its heuristic support; we will only rely on the formulation of 
Conjecture \ref{newspiky}.\end{rem}

\begin{exas}
\label{data}
For each of the three elliptic curves 11A1, 37A1, and 32A1 
(in the notation of Cremona's tables \cite{cremona}), and five (prime) 
values of $d$, we computed the first (approximately) 50,000 normalized 
$\theta$-coefficients $\tilde{c}_{F,\gamma}$, with $F$ ordered by conductor and 
$\gamma$ generic.  The resulting distributions 
are shown in Figures 1 through 3.  As $d$ grows the distributions approach the ``expected'' 
normal distribution, shown as the dashed line in each figure.

\begin{figure}[H]
\includegraphics[height=3.1in]{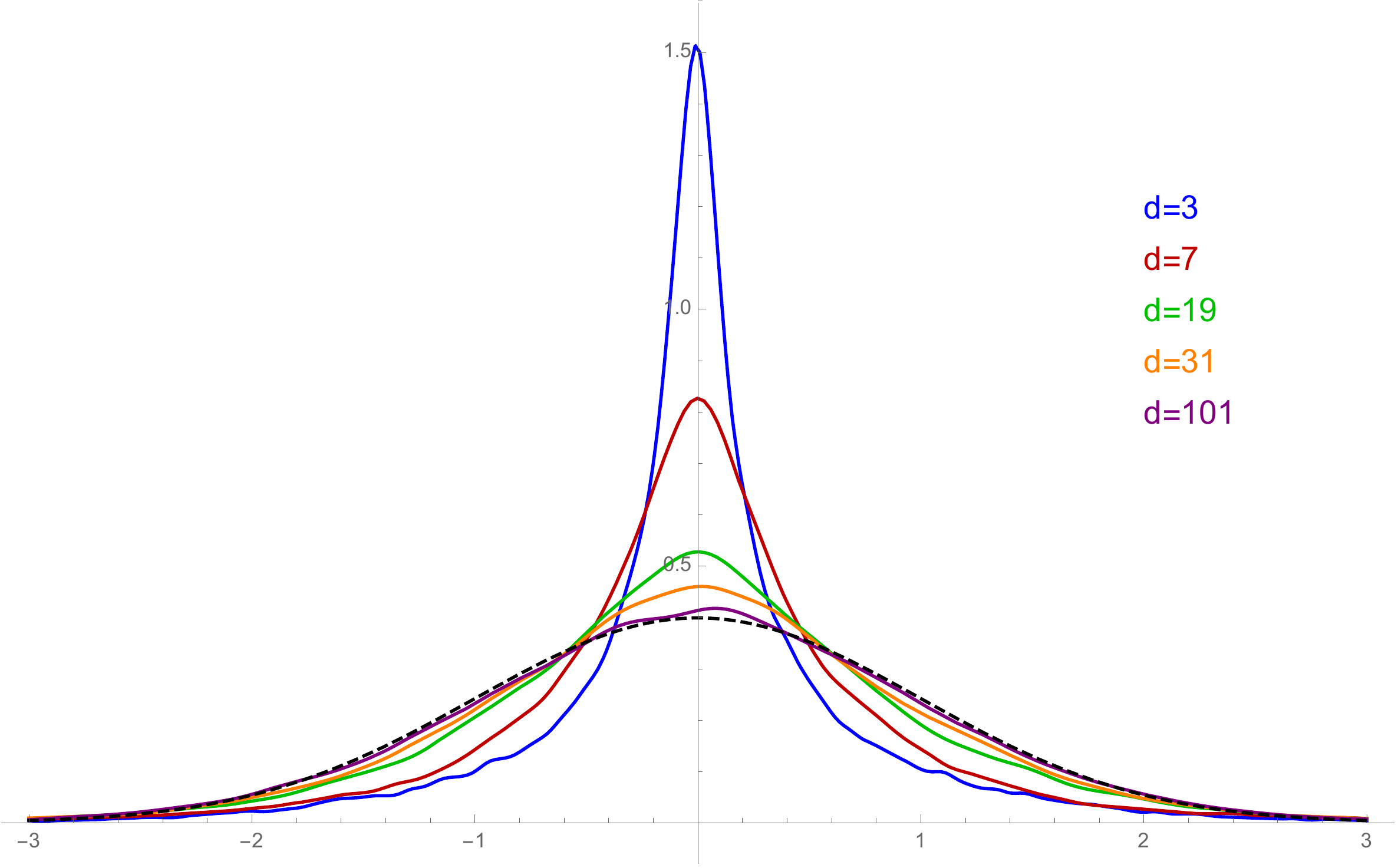}
\caption{Distribution of normalized $\theta$-coefficients for $E$ = 11A1 and varying $d$.}
\end{figure}
\begin{figure}[H]
\includegraphics[height=3.1in]{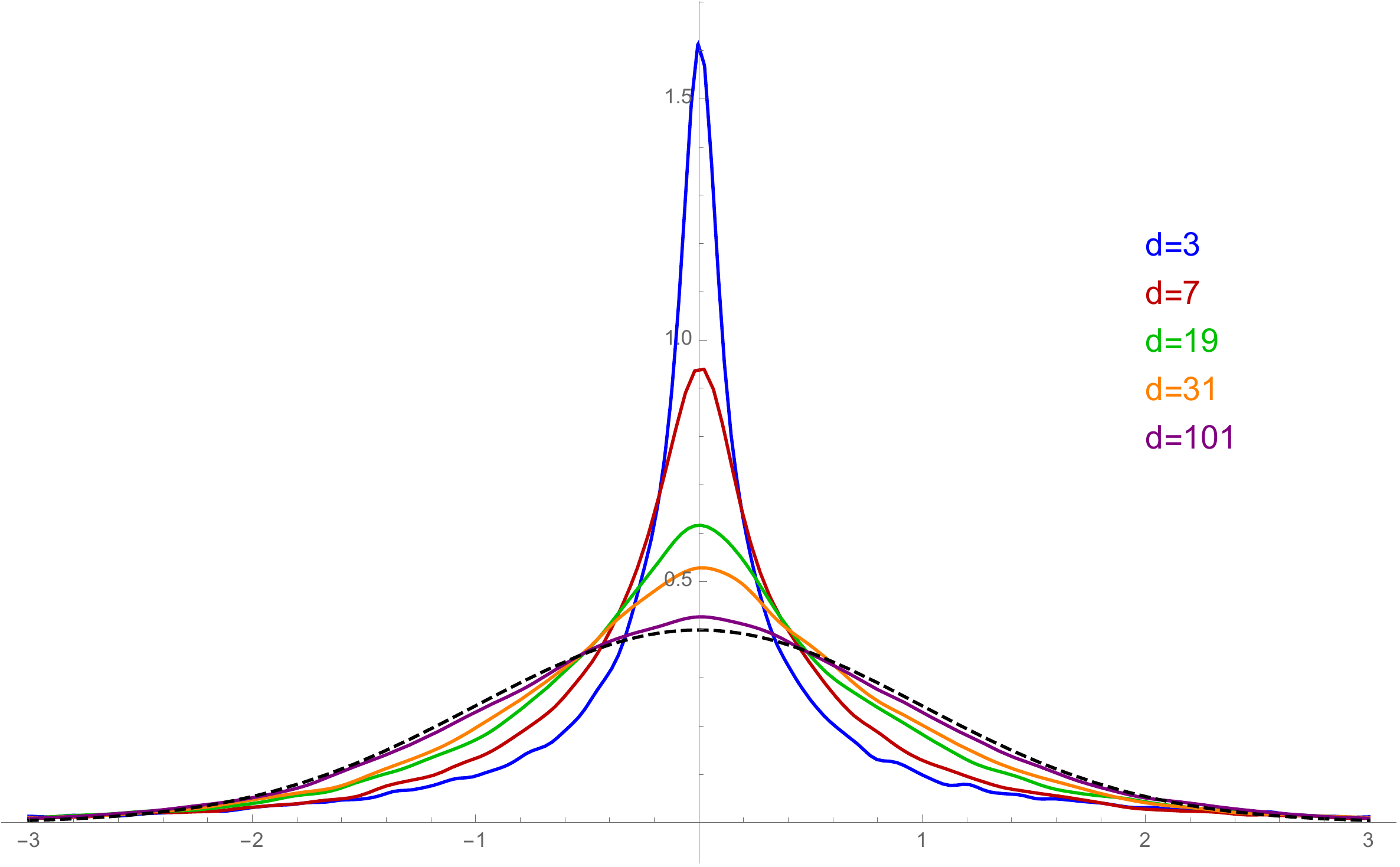}
\caption{Distribution of normalized $\theta$-coefficients for $E$ = 37A1 and varying $d$.}
\end{figure}
\begin{figure}[H]
\includegraphics[height=3.1in]{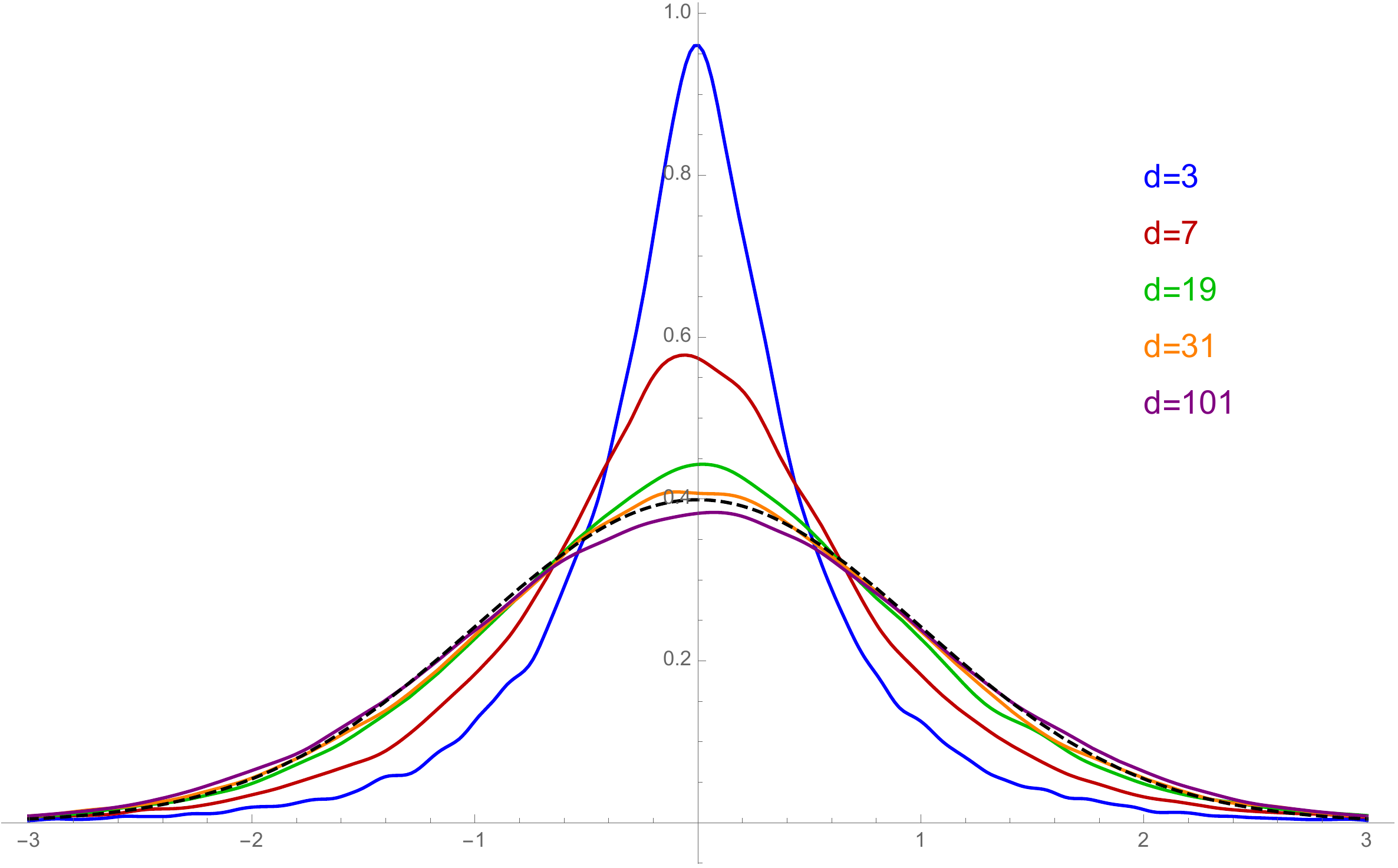}
\caption{Distribution of normalized $\theta$-coefficients for $E$ = 32A1 and varying $d$.}
\end{figure}
\end{exas}

\part{Heuristics}

\section{The heuristic for cyclic extensions of prime degree}
\label{h_prime}

Recall that we have assumed that $E$ is semistable.

Our heuristic is based on two assumptions:
\begin{itemize}
\item[(H1)]
the $c_{F,\gamma}$ are distributed randomly subject to the constraints imposed 
by Conjecture \ref{newspiky} and Remark \ref{prob},
\item[(H2)]
the only correlations among the $c_{F,\gamma}$ for fixed $F$ and varying $\gamma$ 
are the Atkin--Lehner relations of Lemma \ref{alinv}.
See Remark \ref{intra} below for more about the issue of ``intra-correlation''.
\end{itemize}

For  simplicity, we first consider the case of extensions of prime degree.
Fix for this section a character $\chi$ of prime order $p \ge 3$.  As above 
we write $m$ for the conductor of $\chi$, $F$ for the corresponding real cyclic 
extension of $\Q$, and $G := \Gal(F/\Q)$.  Let $B_{E}, \alpha_p, \beta_p \in \R$ be as in 
Conjecture \ref{newspiky}.

Since $p$ is odd, Lemma \ref{alinv}(ii) shows that the involution 
$\iota_F$ has a unique fixed element 
$\sigma \in G$.  Choose a subset $S \subset G$ consisting 
of one element from each pair $\{\gamma,\iota_F(\gamma)\}$ of generic elements.  

\begin{lem}
\label{eqth}
We have
$$
\sum_{\gamma\in G} \chi(\gamma) c_{F,\gamma} = 0 \iff \text{$c_{F,\gamma} = c_{F,\sigma}$ 
for all $\gamma \in S$.}
$$
\end{lem}

\begin{proof}  
The only $\Q$-linear relation among the values of $\chi$ 
(i.e., the $p$-th roots of unity) is that their 
sum is zero.  It follows that the sum over $\gamma$ is zero if and only if 
all $c_{F,\gamma}$ are equal.
The lemma now follows from Lemma \ref{alinv}(i) (note that $a_{\sigma} = 0$ if $w_F = -1$).
\end{proof}

\begin{heur}
\label{hh}
With notation as above, 
$$
\Prob[L(E,\chi,1) = 0] 
   < \biggl(B_{E}^2 \frac{p m^{2\alpha_p}}{\varphi(m)\log(m)^{1-2\beta_p}}\biggr)^{(p-1)/4}.
$$
\end{heur}

\begin{rem}
The notation $\Prob[L(E,\chi,1) = 0]$ is meant to stand for (an upper bound for) 
the heuristic likelihood that the value $L(E,\chi,1)$ vanishes (see Remark \ref{prob}). 
\end{rem}

\begin{proof}[Justification for Heuristic \ref{hh}]
By Corollary \ref{thetaL} and Lemma \ref{eqth}, 
$$
L(E,\chi,1) = 0 \iff \sum_{\gamma\in G} \chi(\gamma) c_{F,\gamma} = 0
   \iff \text{$c_{F,\gamma} = c_{F,\sigma}$ for all $\gamma\in S$.}
$$
Let 
$$
r_\chi := \frac{\sqrt{p}}{\sqrt{\varphi(m)\log(m)}},
$$
so the normalized $\theta$-coefficients $\tilde{c}_{F,\gamma}$ of 
Definition \ref{6.1} are given by 
$\tilde{c}_{F,\gamma} = c_{F,\gamma} r_\chi$.
Since the $c_{F,\gamma}$ are integers, we have 
\begin{multline*}
c_{F,\gamma} = c_{F,\sigma} \iff c_{F,\gamma} \in 
   (c_{F,\sigma}-\textstyle\frac{1}{2}, c_{F,\sigma}+\frac{1}{2}) \\
   \iff \tilde{c}_{F,\gamma} \in 
      \textstyle \bigl(\tilde{c}_{F,\sigma}-\frac{r\chi}{2}, 
      \tilde{c}_{F,\sigma}+\frac{r\chi}{2}\bigr)
\end{multline*}
Using Conjecture \ref{newspiky}, Remark \ref{prob} 
and our heuristic assumptions (H1) and (H2) above, we have
\begin{multline*}
\Prob[c_{F,\gamma} = c_{F,\sigma} ~\forall \gamma\in S] 
   = \prod_{\gamma\in S} \Prob[\tilde{c}_{F,\gamma} \in 
   \textstyle (\tilde{c}_{F,\sigma}-\frac{r_\chi}{2}, \tilde{c}_{F,\sigma}+\frac{r_\chi}{2})]  \\
< (B_E r_\chi m^{\alpha_p} \log(m)^{\beta_p})^{(p-1)/2}
= \biggl(B_{E}^2 \frac{p m^{2\alpha_p} \log(m)^{2 \beta_p}}{\varphi(m)\log(m)}\biggr)^{(p-1)/4}
\end{multline*}
as desired.
\end{proof}

\section{The heuristic for general cyclic extensions}
\label{h_general}

Fix for this section an even character $\chi$ of arbitrary degree $d \ge 3$.  As above 
we write $m$ for the conductor of $\chi$, $F$ for the corresponding real cyclic 
extension of $\Q$, and $G := \Gal(F/\Q)$.  Let $B_{E}$, $\alpha_d$ and $\beta_d$
be the constants in Conjecture \ref{newspiky}.

\begin{heur}
\label{h2}
With notation as above,
$$
\Prob[L(E,\chi,1) = 0] < \biggl(B_{E}^2 \frac{d m^{2\alpha_d}}{\varphi(m)\log(m)^{1-2\beta_d}}\biggr)^{\varphi(d)/4}.
$$
\end{heur}

The rest of this section is devoted to justifying Heuristic \ref{h2}, using 
heuristic assumptions (H1) and H(2) of \S\ref{h_prime}.

Recall the sign $w_F = \pm 1$, the element $\gamma_F \in G$, and the involution 
$\iota_F : G \to G$ of Definition \ref{sensitive}.
We extend $\iota_F$ to a $\Z$-linear involution of $\Z[G]$ 
(also denoted by $\iota_F$).

Fix a generator $g$ of $G$.  Let $\zeta \in \bmu_d$ denote the primitive 
$d$-th root of unity $\chi(g)$.  
Define a $\Z$-linear involution $\iota_\chi$ of the cyclotomic ring $\Z[\zeta]$ by 
$$
\iota_\chi(\rho) = \chi(\gamma_F)^{-1}\bar\rho
$$
where $\rho \mapsto \bar\rho$ is complex conjugation.  
Then we have a commutative diagram
$$
\xymatrix{
\Z[G]^{\iota_F=w_F} \ar@{^(->}[r] \ar@{->>}_\chi[d] 
   & \Z[G] \ar^{\iota_F}[r] \ar@{->>}_\chi[d] & \Z[G] \ar@{->>}^\chi[d] \\
\Z[\zeta]^{\iota_\chi=w_F} \ar@{^(->}[r] & \Z[\zeta] \ar_{\iota_\chi}[r] & \Z[\zeta]
}
$$
where, for example, $\Z[G]^{\iota_F=w_F}$ denotes 
$\{\rho \in \Z[G] : \iota_F(\rho) = w_F \rho\}$.

\begin{defn}
\label{bases}
For $\gamma \in G$, define $v_\gamma \in \Z[G]^{\iota_F=w_F}$ by
$$
v_\gamma := 
\begin{cases}
\gamma + w_F \iota_F(\gamma) & \text{if $\gamma$ is generic}, \\
\gamma & \text{if $\gamma$ is special$^+$}, \\
0 & \text{if $\gamma$ is special$^-$}.
\end{cases}
$$
\end{defn}

Fix a subset $G_0 \subset G$ consisting of all special$^+$ elements of $G$ 
and one element from each pair $\{\gamma,\iota_F(\gamma)\}$ of generic elements.  
Then the set $\{v_\gamma : \gamma \in G_0\}$ is a $\Z$-basis of $\Z[G]^{\iota_F=w_F}$.

Since $\chi : \Z[G]^{\iota_F=w_F} \to \Z[\zeta]^{\iota_\chi=w_F}$ 
is surjective, we can choose a subset $S \subset G_0$ such that, 
writing $A \subset \Q[G]^{\iota_F=w_F}$ for the $\Q$-vector space spanned by 
$\{v_\gamma : \gamma \in S\}$, 
$\chi$ maps $A$ isomorphically to $\Q(\zeta)^{\iota_\chi=w_F}$.  
Let $\lambda : \Q(\zeta)^{\iota_\chi=w_F} \to A$ be the inverse isomorphism.
In particular we have
$$
\#S = \dim_\Q\Q(\zeta)^{\iota_\chi=w_F} = \varphi(d)/2.
$$

\begin{proof}[Justification for Heuristic \ref{h2}] 
Define 
$$
\rho := \sum_{\gamma\in G} c_{F,\gamma} \gamma.
$$
Then $\rho \in  \Z[G]^{\iota_F=w_F}$, so $\rho = \sum_{\gamma\in G_0}c_{F,\gamma} v_\gamma$.
Let 
$$
\rho_1 := \sum_{\gamma\in S}c_{F,\gamma} v_\gamma, \qquad  
   \rho_2 := \sum_{\gamma\in G_0-S}c_{F,\gamma} v_\gamma,
$$
so $\rho = \rho_1+\rho_2$ and $\rho_1 \in A$.
We have
\begin{multline*}
\sum_{\gamma\in G} \chi(\gamma) c_{F,\gamma} = 0 \iff \chi(\rho) = 0 \\
   \iff \chi(\rho_1) = -\chi(\rho_2)
   \iff \rho_1 = -\lambda(\chi(\rho_2)).
\end{multline*}
Since $\lambda(\chi(\rho_2)) \in A$ we can write 
$-\lambda(\chi(\rho_2)) = \sum_{\gamma\in S}b_\gamma v_\gamma$ with $b_\gamma \in \Q$.
Hence, writing $r_\chi := \frac{\sqrt{d}}{\sqrt{\varphi(m)\log(m)}}$ as in \S\ref{h_prime},
\begin{align*}
\sum_{\gamma\in G} \chi(\gamma) c_{F,\gamma} = 0 
   &\iff c_{F,\gamma} = b_\gamma \quad \text{for every $\gamma\in S$} \\
   &\iff c_{F,\gamma} \in \textstyle(b_\gamma-\frac{1}{2}, b_\gamma+\frac{1}{2}) 
   \quad \text{for every $\gamma\in S$} \\
   &\iff \tilde{c}_{F,\gamma} 
   \in \textstyle({b_\gamma}{r_\chi}-\frac{r_\chi}{2}, {b_\gamma}{r_\chi}+\frac{r_\chi}{2}) 
   \quad \text{for every $\gamma\in S$} 
\end{align*}
so using Conjecture \ref{newspiky}, Remark \ref{prob} 
and the heuristic assumptions (H1) and (H2) of \S\ref{h_prime} we have
\begin{multline*}
\Prob\biggl[\sum_{\gamma\in G} \chi(\gamma) c_{F,\gamma} = 0\biggr] 
   = \prod_{\gamma\in S} \Prob\biggl[\tilde{c}_{F,\gamma} 
      \in \textstyle({b_\gamma}{r_\chi}-\frac{r_\chi}{2}, {b_\gamma}{r_\chi}+\frac{r_\chi}{2}) \biggr] \\
   < (B_{E} r_\chi m^{\alpha_d} \log(m)^{\beta_d})^{\#S} 
   = \biggl(B_E^2 \frac{d m^{2\alpha_d}\log(m)^{2\beta_d}}{\varphi(m)\log(m)}\biggr)^{\#S/2}
\end{multline*}
(note that the $b_\gamma$ for $\gamma \in S$ depend only on the $c_{F,\gamma}$ for $\gamma \in G_0-S$).
Since $\#S = \varphi(d)/2$, and 
$L(E,\chi,1) = 0$ if and only if $\sum_{\gamma\in G} \chi(\gamma) c_{F,\gamma} = 0$ by Corollary \ref{thetaL}, 
this gives the desired heuristic.
\end{proof}

\begin{rem}
\label{intra}
The upper bound of Heuristic \ref{h2}, specifically the exponents $\varphi(d)/4$ and 
$\varphi(d)/2$, depends on the assumption (H2) of no intra-correlations, i.e., no 
statistical correlations connecting $\varphi(d)/2$ different $\theta$-coefficients 
of a given $\theta$-element, these being chosen to have the property that no two
are brought into one another by the involution $\iota_F$ (see \S\ref{dt}).   
We might refer to the exponent $\varphi(d)/2$ as the {\em intra-correlation exponent}.  
We will see that a qualitative version of our heuristic will still seem viable even if 
we assume significantly {\em smaller} intra-correlation exponents, e.g., if we 
have an upper bound of the form:
$$
\Prob[L(E,\chi,1) = 0] < 
   \biggl(B_{E}^2 \frac{d m^{2\alpha_d}}{\varphi(m)\log(m)^{1-2\beta_p}}\biggr)^{t(d)}
$$
where $t(d) \gg \log(d)$ (see Proposition \ref{lemprop2} below).
\end{rem}

\section{Expectations}

In this section we suppose that Conjecture \ref{newspiky} holds, and we 
fix data $B_E, \alpha_d, \beta_d$ for all $d \ge 3$ as in that conjecture.

\begin{defn}
\label{cd}
For every $d$ let $\X_d$ denote the set of even primitive Dirichlet characters of order $d$.
Define 
$$
F_d(m) := \#\{\chi\in \X_d : \mathrm{conductor}(\chi) = m\},
$$
and for every $X$
$$
\X_d(X) := \{\chi\in\X_d : \mathrm{conductor}(\chi) \le X\},
$$
\begin{equation}
\label{Ed}
\E_d(X) := \sum_{m\le X} F_d(m) 
   \biggl(B_{E}^2 \frac{d m^{2\alpha_d}}{\varphi(m)\log(m)^{1-2\beta_d}}\biggr)^{\varphi(d)/4}.
\end{equation}
\end{defn}

Our Heuristic \ref{h2} suggests the following.
\begin{heur}
\label{heurist}
We expect
$$
\#\{\chi \in \X_d(X) : L(E,\chi,1) = 0\} \le \E_d(X).
$$
In particular if $\E_d(\infty)$ is finite, then the number of $\chi$ of 
order $d$ with $L(E,\chi,1) = 0$ should be finite.

More generally, if $D \subset \Z_{\ge 3}$, then we expect
$$
\#\{\chi \in \cup_{d \in D}\X_d(X) : L(E,\chi,1) = 0\} \le \sum_{d \in D}\E_d(X).
$$
\end{heur}

\begin{prop}
\label{expprop}
Suppose $p$ is prime, and let 
$$
\sigma := (1-2\alpha_p)(p-1)/4, \quad \tau := \beta_p(p-1)/2.
$$
As $X \in \R$ goes to infinity we have the following upper bounds on $\E_p(X)$.
\begin{enumerate}
\itemsep=3pt
\item
If $\sigma>1$, or $\sigma=1$ and $\tau<-1$, then $\E_p(\infty)$ is finite.
\item
If $\sigma = 1$ and $\tau=-1$, then $\E_p(X) \ll_{E,p} \log\log(X)$.
\item
If $\sigma = 1$ and $\tau>-1$, then 
$
\E_p(X) \ll_{E,p} \log(X)^{\tau+1}.
$
\item
If $0 < \sigma < 1$, then
$
\E_p(X) \ll_{E,p} X^{1-\sigma} \log(X)^\tau.
$
\end{enumerate}
\end{prop}

\begin{proof}
This follows from Lemma \ref{gr} below.  Precisely, since 
$m \ll \varphi(m)\log(m)$ (see for example \cite[Theorem 328]{h-w}), 
we have
$$
\E_p(X) \ll_{E,p}
   \sum_{m \le X} F_p(m) \frac{\log(m)^\tau}{m^{\sigma}}.
$$
Let $g(x) := \log(x)^\tau/x^\sigma$.  Then on a suitable interval $[r,\infty)$, 
$g$ satisfies the hypotheses of Lemma \ref{gr}, so that lemma shows that
$$
\E_p(X) \ll_{E,p} \int_r^X g(t) dt.
$$
Now the proposition follows.
\end{proof}

The following conjectures and predictions are suggested by the random 
matrix heuristic.  This involves work of  Conrey, Keating, Rubinstein,
and Snaith \cite{CKRS} for 
$p=2$, and David, Fearnley, and Kisilevsky \cite{DFK,DFK2}  and  
Fearnley,  Kisilevsky, and Kuwata \cite{FKK} for $p \ge 3$.  
For more precise references see Remark \ref{10.6} below.

If $p(X), q(X) : \R_{\ge r} \to \R_{>0}$ are functions for some real number $r$, we write 
$p(X) \sim q(X)$ to mean that $\lim_{X \to \infty} p(X)/q(X)$ exists and is nonzero.

\begin{conj}[\cite{DFK,DFK2, FKK}]
\label{dfkconj}
For a prime $p \ge 3$ let 
$$
n_p(X) := \#\{\chi\in\X_p(X) : L(E,\chi,1) = 0\}.
$$ 

Then there are nonzero constants $c_{E,p}$ such that
\begin{enumerate}
\item $n_3(X) \sim {\sqrt X}\log(X)^{c_{E,3}}$,
\item $n_5(X) \sim \log(X)^{c_{E,5}}$,
\item
$n_p(X)$ is bounded independently of $X$ if $p \ge 7$.
\end{enumerate}
\end{conj}

\begin{rem}
\label{10.6}
Assertion (iii) is part of \cite[Conjecture 1.2]{DFK2}.  

The statements (i) and (ii) are implicit in \cite[p. 256]{DFK2} but 
not explicitly stated as conjectures.  The authors of \cite{DFK2} comment:
\begin{quote}  \sl The exact power of log X that is obtained with the
random matrix approach depends subtly on the discretisation, and is
difficult to predict.\end{quote}
They state a weaker conjecture:
\begin{conj}[Conjecture 1.2 of \cite{DFK2}] 
~
\begin{enumerate}
\item 
$\lim_{X \to \infty}\log n_3(X)/\log(X) = 1/2.$
\item
$n_5(X)$ is unbounded and $\ll X^\epsilon$  as $X$ tends to infinity, 
for any $\epsilon >0$.
\end{enumerate}
\end{conj}
\end{rem}

\begin{rem}      
See \cite{FKK} for further interesting results related to the above conjecture for cubic characters.
Fearnley and  Kisilevsky  \cite{FK} exhibit one example of $L(E,\chi,1)=0$ with 
$\chi$ a character of order $\ell = 11$ and $E$ the elliptic curve $5906B1$ (using 
Cremona's classification \cite{cremona}).  The character $\chi$ is of conductor $23$ (i.e., $\chi$ 
has the smallest possible conductor for characters of its order).  
In this example $\rk(E(\Q)) = 2$, and $\rk(E(\Q(\bmu_{23})^+)) = 12$.
\end{rem}
   
\begin{rem}
Suppose $p=3$, $\alpha_3 = 0$, and $\beta_3 = c_{E,3}$.
Then Proposition \ref{expprop}(iv) gives
$$
\E_3(X) \ll X^{1/2} \log(X)^{c_{E,3}}.
$$

Now suppose $p=5$, $\alpha_5 = 0$, and $\beta_5 = (c_{E,5}-1)/2$.
Then Proposition \ref{expprop}(iii) gives
$$
\E_5(X) \ll \log(X)^{c_{E,5}}.
$$

Now suppose $p\ge 7$ and $\alpha_p < (p-5)/(2(p-1))$.  
Then Proposition \ref{expprop}(i) shows that $\E_p(\infty)$ is finite.

In other words, for appropriate values of $\alpha$ and $\beta$ (that are consistent with 
computational data), Heuristic \ref{heurist} is consistent with Conjecture \ref{dfkconj}.
\end{rem}

\begin{prop}
\label{basic}
We have
$$
\sum_{d \,:\, \varphi(d)(1-2\alpha_d)>4} \E_d(\infty) < \infty.
$$
\end{prop}

\begin{proof}
This follows from Proposition \ref{lemprop2} below (and comparing equations 
\eqref{Ed} and \eqref{Ed2})  taking  
$$
t(d) := \varphi(d)/4, \quad 
   \sigma_d := 2\alpha_d, \quad
   \tau_d := 2\beta_d, \quad
   M := B_E^2.  
$$
\end{proof}

\section{Conjectures}\label{CONJ}
Proposition \ref{basic}, Heuristic \ref{heurist}, and Conjecture \ref{newspiky} 
lead to the following ``analytic" and ``arithmetic'' conjectures.  
\begin{conj}
\label{bigc}
Let
$$
\X = \bigcup_{\varphi(d)>4} \X_d
$$
(the set of all even Dirichlet characters of order at least $7$ and different 
from $8$, $10$, or $12$).
Then the set
$
\{ \chi \in \X : L(E,\chi,1) = 0\}
$
is finite. 
\end{conj}

\begin{conj}\label{conj235}
Suppose $F/\Q$ is an (infinite) real abelian extension  that 
has only finitely many subfields of degree $2$, $3$, or $5$.  Then
$E(F)$ is finitely generated.
\end{conj}

\noindent
{\em Conjecture  \ref{conj235} would be a consequence of Conjecture \ref{bigc} 
and the Birch and Swinnerton-Dyer conjecture.}  
As in Proposition \ref{bsdcor}, if the Birch and Swinnerton-Dyer conjecture
holds then it is enough to show that there are only finitely many 
characters $\chi$ of $\Gal(F/\Q)$ such that $L(E,\chi,1) = 0$.  
But the hypotheses in Conjecture \ref{conj235} imply that $\Gal(F/\Q)$ 
has only finitely many characters of order $2$, $3$, or $5$, 
and consequently it has only finitely many characters of order $6$, $8$, $10$, or $12$ as well. 
Now the conjecture follows from Conjecture \ref{bigc}.
\hfill \qedsymbol

\begin{rem}  As mentioned in Remark \ref{intra}, our heuristic---following 
Proposition \ref{lemprop2}---would suggest qualitatively similar conjectures
(possibly with a larger set of exceptional degrees $d$)
even if there were significant ``intra-correlation".
 \end{rem}

\section{Analytic results}
\label{analyticsect}

\begin{defn}
\label{sumd}
For every $d \ge 1$, recall (Definition \ref{cd}) that $F_d(m)$
denotes the number of even primitive Dirichlet characters of order $d$ and conductor $m$.
We are interested in sums of the form 
$$
S_d(g;X) := \sum_{m=r}^X F_d(m) g(m)
$$
where $g : \R_{\ge r} \to \R_{\ge 0}$ is a function, and we let $S_d(g) = S_d(g;\infty)$.

Let $\FF_d(m)$ denote the number of {\em all} primitive Dirichlet characters of 
order dividing $d$ and conductor $m$.  The following theorem is essentially due to Kubota \cite{K}.
\end{defn}

\begin{thm}[Kubota]
\label{22}
For every $d \ge 1$, define Dirichlet series
$$
L_d(s) := \sum_{m=1}^\infty {\FF_d(m)}{m^{-s}}, \quad 
   Z_d(s) := \prod_{t \mid d, t>1} \zeta_{\Q(\bmu_t)}(s)
$$
where $\zeta_{\Q(\bmu_t)}$ denotes the Dedekind $\zeta$-function of the 
cyclotomic field $\Q(\bmu_t)$.  
Then $L_d(s)/Z_d(s)$ has an analytic continuation to the half plane $\Re(s) > 1/2$ 
and is nonzero at $s=1$.
\end{thm}

\begin{proof}
The key ideas of the proof are in \cite[Theorem 6]{K}, although Kubota 
does not state a result this strong.  See also \cite[Proposition 5.2]{DFK} 
for a proof in the case $d = 3$ that works for all primes $d$.  
We prove the theorem in general by showing that $L_d(s)/Z_d(s)$ is given by 
an Euler product that converges on the half plane $\Re(s) > 1/2$.

Suppose $m \in \Z_{>0}$ factors into powers of distinct primes as $m = \prod_i q_i$.  
Using the isomorphism $(\Z/m\Z)^\times \cong \prod_i (\Z/q_i\Z)^\times$, we see that
every Dirichlet character of conductor $m$ and order dividing $d$ can be written 
uniquely as a product of Dirichlet characters of conductor $q_i$ and order dividing $d$.  
It follows that 
the function $\FF_d(m)$ is multiplicative in $m$, and thus $L_d(s)$ can be written 
as an Euler product
$$
L_d(s) = \prod_p L_{d,p}(s) \quad \text{where} \quad 
   L_{d,p}(s) := \sum_{k=0}^\infty \FF_d(p^k)p^{-ks}.
$$
It is straightforward to verify that: 
\begin{itemize} 
\item
$L_{d,p}(s)$ is a (finite degree) polynomial in $p^{-s}$,
\item
$\FF_d(p) = \gcd(d,p-1)-1$,
\item 
if $p \nmid d$ and $k \ge 2$ then $\FF_d(p^k) = 0$.
\end{itemize}
In particular 
\begin{equation}
\label{x1}
L_{d,p}(s) = 1 + (\gcd(d,p-1)-1)p^{-s} \quad\text{if $p \nmid d$}.
\end{equation}

Let $\cL_{d,p}(s)$ denote the Euler factor at $p$ of the Dirichlet series $Z_d(s)^{-1}$.  
If $t \in \Z_{>0}$ and $p \nmid t$ let $\nu(p,t)$ denote the order of $p$ in $(\Z/t\Z)^\times$.  
Then if $p \nmid d$ we have
\begin{multline*}
(1-p^{-s})\cL_{d,p}(s) = \prod_{t\mid d}(1-p^{-\nu(p,t)s})^{\varphi(t)/\nu(p,t)} \\[-5pt] 
   = \prod_{n \mid d} (1-p^{-ns})^{\sum_{t \mid d, \nu(p,t)=n}\varphi(t)/n}. 
\end{multline*}
The exponent of the $n=1$ term is $\sum_{t \mid d, t \mid (p-1)}\varphi(t) = \gcd(d,p-1)$, so
\begin{equation}
\label{x2}
\cL_{d,p}(s) = (1-p^{-s})^{\gcd(d,p-1)-1}
   \prod_{n \mid d,n>1} (1-p^{-ns})^{\sum_{t \mid d, \nu(p,t)=n}\varphi(t)/n}.
\end{equation}
It follows from \eqref{x1} and \eqref{x2} that the product 
$\prod_{p}L_{d,p}(s) \cL_{d,p}(s)$ converges on the half plane $\Re(s) > 1/2$ 
and is not zero at $s=1$.  This completes the proof of the theorem.
\end{proof}

\begin{cor}
\label{22b}
The Dirichlet series $L_d(s)$ of Theorem \ref{22} converges on the half plane 
$\Re(s) > 1$, and has a pole of order $\tau(d)-1$ at $s=1$, where $\tau(d)$ 
is the number of divisors of $d$.
\end{cor}

\begin{proof}
This is immediate from Theorem \ref{22}, since each $\zeta_{\Q(\bmu_t)}$ has a simple pole at $s=1$.
\end{proof}

Since $F_d(m) \le \FF_d(m)$, the following is an immediate consequence of Corollary \ref{22b}.

\begin{cor}
\label{s1c}
If $g(X) \ll X^{-s}$ for some $s > 1$, then $S_d(g)$ is finite.
\end{cor}

Recall that if $p(X), q(X) : \R_{\ge r} \to \R_{>0}$ are functions for some real number $r$, we write 
$p(X) \sim q(X)$ to mean that $\lim_{X \to \infty} p(X)/q(X)$ exists and is nonzero.

\begin{lem}
\label{gr}
Suppose $r \in \R$ and $g : \R_{\ge r} \to \R_{> 0}$ is decreasing, 
continuously differentiable, and satisfies $-g'(X) \ll g(X)/X$.  
Then for every prime $p$ we have
$$
S_p(g;X) \ll_{p,g} \sum_{m=r}^X g(m) \sim \int_{r}^X g(t) dt.
$$
\end{lem}

\begin{proof}
Fix a prime $p$, and let $A(X) := \sum_{m=r}^X \FF_p(m)$.  
Since $p$ is prime, Corollary \ref{22b} shows that the Dirichlet series $L_p(s)$ of 
Theorem \ref{22} has a simple pole at $s=1$.  Hence a standard Tauberian theorem 
(for example \cite[Theorem I, Appendix 2]{nark}) that 
\begin{equation}
\label{S3}
A(X) = cX + o(X)
\end{equation}
where $c$ is the residue at $s=1$ of $L_p(s)$.

Let 
$$
\SSS_p(g;X) := \sum_{m=r}^X \FF_p(m) g(m).
$$
Abel summation (see for example \cite[Theorem 4.2]{apostol}) gives
\begin{equation}
\label{AS}
\SSS_p(g;X) = A(X)g(X) - A(r)g(r) - \int_r^X A(t)g'(t)dt.
\end{equation}
Integration by parts gives
\begin{equation}
\label{IP}
 \int_r^X g(t) dt = X g(X) - r g(r) - \int_r^X t g'(t) dt.
\end{equation}
Let $B(X) := A(X) - cX$ with $c$ as in \eqref{S3}.
Subtracting $c$ times \eqref{IP} from \eqref{AS} gives
\begin{multline}
\label{IPAS}
\SSS_p(g;X) - c\int_r^X g(t) dt \\ = B(X) g(X) - B(r) g(r) - \int_r^X B(t) g'(t)dt.
\end{multline}

Since $g$ is decreasing and $-g'(X) \ll g(X)/X$, we have
\begin{equation}
\label{o1}
\biggl|\int_r^X B(t) g'(t)dt \biggr| \le -\int_r^X |B(t)| g'(t)dt 
   \ll \int_r^X \frac{|B(t)|}{t} g(t)dt.
\end{equation}
In addition, since $B(X) = o(X)$ and $g$ is decreasing,
\begin{equation}
\label{doub}
B(X) g(X) = o(X g(X)) \quad\text{and}\quad Xg(X) \le \int_r^X g(t)dt + r g(r).
\end{equation}

Combining \eqref{IPAS} with \eqref{o1} and \eqref{doub} 
shows that 
$$
\SSS_p(g;X) - c\int_r^X g(t) dt = 
   \begin{cases} 
   o\bigl(\int_r^X g(t) dt\bigr) & \text{if $\int_r^X g(t)dt$ is unbounded,} \\[3pt]
   O(1) & \text{if $\int_r^X g(t)dt$ is bounded.}
   \end{cases}
$$
Since $S_p(g,X) \le \SSS_p(g,X)$, this proves the lemma.
\end{proof}

\begin{lem}
\label{hw2}
There is an absolute constant $C$ such that for every $m \ge 3$ we have 
$\varphi(m) \log\log(m) > m/C$.
\end{lem}

\begin{proof}
This follows from \cite[Theorem 328]{h-w}.
\end{proof}

\begin{prop}
\label{lemprop2}
Suppose $M \in \R_{>0}$, 
$t : \Z_{\ge 3} \to \R$ is a function satisfying $t(d) \gg \log(d)$, 
and $\sigma_d, \tau_d$ are sequences of real numbers 
such that $\{t(d)\sigma_d : d \ge 3\}$ is bounded and $\lim_{d\to\infty}\tau_d = 0$.  Then
\begin{equation}
\label{Ed2}
\sum_{d \,:\, t(d)(1-\sigma_d) > 1} \sum_{m=1}^\infty F_d(m) 
   \biggl(\frac{M d m^{\sigma_d} \log(m)^{\tau_d}}{\varphi(m)\log(m)}\biggr)^{t(d)}
\end{equation}
converges.
\end{prop}

\begin{proof}
Let
$$
T_d := \sum_{m=1}^\infty F_d(m) 
   \biggl(\frac{M d m^{\sigma_d} \log(m)^{\tau_d}}{\varphi(m)\log(m)}\biggr)^{t(d)}.
$$
We will show below that $T_d$ converges if $t(d)(1-\sigma_d) > 1$.  
We want to show that in fact $\sum_{d \, : \, t(d)(1-\sigma_d)>1} T_d$ converges.
Note that $F_d(m) = 0$ unless $m > d$.  Fix $k$ such that if $d \ge k$, then 
\begin{equation}
\label{items}
\parbox{4.2in}{
\begin{itemize}
\item
$t(d)(1-\sigma_d) > 3$, 
\item
$\tau_d < 1/2$,
\item
$\log\log(x)/\log(x)^{1/2}$ is decreasing on $[d,\infty)$.
\end{itemize}}
\end{equation}
Let $C$ be the positive constant from Lemma \ref{hw2}.
Using Lemma \ref{hw2}, we have for every $d$ 
\begin{multline}
\label{sk1}
T_d = (Md)^{t(d)} \sum_{m > d} 
      F_d(m) \biggl(\frac{m^{\sigma_d}}{\varphi(m) \log(m)^{(1-\tau_d)}}\biggr)^{t(d)} \\
   < (Md)^{t(d)} \sum_{m > d}
      F_d(m)\biggl(\frac{C\log\log(m)}{m^{1-\sigma_d}\log(m)^{1-\tau_d}}\biggr)^{t(d)}.
\end{multline}
We have $F_d(m) \le |\Hom((\Z/m\Z)^\times,\bmu_d)| < m$, and using \eqref{items} we get 
\begin{align}
\notag
   \sum_{d \ge k} T_d 
   &< \sum_{d \ge k} (Md)^{t(d)} \sum_{m > d}
      m\biggl(\frac{C\log\log(m)}{m^{1-\sigma_d}\log(m)^{1/2}}\biggr)^{t(d)} \\
\notag
   &<\sum_{d \ge k} \biggl(\frac{MCd\log\log(d)}{\log(d)^{1/2}}\biggr)^{t(d)} 
      \sum_{m > d} m^{1-t(d)(1-\sigma_d)} \\
\notag
   &< \sum_{d \ge k}\biggl(\frac{MCd\log\log(d)}{\log(d)^{1/2}}\biggr)^{t(d)} 
      \frac{d^{2-t(d)(1-\sigma_d)}}{t(d)(1-\sigma_d)-2} \\
\label{finalsum}
   &\le \sum_{d \ge k} \biggl(\frac{MC\log\log(d)}{\log(d)^{1/2}}\biggr)^{t(d)}
      d^{2+b}
\end{align}
where $b = \sup\{t(d)\sigma_d : d \ge 3\}$.

Fix an integer $n \ge b+4$.  
Since $t(d) \gg \log(d)$, there is a positive constant $C'$ such that for all 
large $d$ we have
\begin{multline*}
\biggl(\frac{\log(d)^{1/2}}{MC\log\log(d)}\biggr)^{t(d)} \ge \log(d)^{C'\log(d)} \\
   = e^{C'\log(d)\log\log(d)} = d^{C'\log\log(d)} > d^n.
\end{multline*}
Thus the sum \eqref{finalsum}  converges.

For the finitely many $d < k$ with $t(d)(1-\sigma_d) > 1$, \eqref{sk1} and 
Corollary \ref{s1c} applied with 
$$
g(x) = \biggl(\frac{\log\log(x)}{x^{1-\sigma_d}\log(x)^{1-\tau_d}}\biggr)^{t(d)}
$$
shows that $T_d$ converges. This completes the proof of the proposition.
\end{proof}

\end{document}